\documentclass[12pt]{amsart}
\usepackage{amssymb}
\usepackage{amsfonts}
\usepackage{amssymb,latexsym}
\usepackage{enumerate}
\usepackage{url}
\usepackage{mathrsfs}
\makeatletter
\@namedef{subjclassname@2010}{%
  \textup{2010} Mathematics Subject Classification}
\makeatother

  [2002/01/19 v2.2g %
    AMS font definitions%
  ]
\DeclareFontFamily{U}{euf}{}
\DeclareFontShape{U}{euf}{m}{n}{%
  <5><6><7><8><9>gen*eufm%
  <10><10.95><12><14.4><17.28><20.74><24.88>eufm10%
  }{}
\DeclareFontShape{U}{euf}{b}{n}{%
  <5><6><7><8><9>gen*eufb%
  <10><10.95><12><14.4><17.28><20.74><24.88>eufb10%
  }{}

  [2002/01/19 v2.2g %
    AMS font definitions%
  ]
\DeclareFontFamily{U}{msb}{}
\DeclareFontShape{U}{msb}{m}{n}{%
  <5><6><7><8><9>gen*msbm%
  <10><10.95><12><14.4><17.28><20.74><24.88>msbm10%
  }{}

  [2002/01/19 v2.2g %
    AMS font definitions%
  ]
\DeclareFontFamily{U}{msa}{}
\DeclareFontShape{U}{msa}{m}{n}{%
  <5><6><7><8><9>gen*msam%
  <10><10.95><12><14.4><17.28><20.74><24.88>msam10%
  }{}

\newtheorem{theorem}{Theorem}[section]
\newtheorem{lemma}[theorem]{Lemma}

\newtheorem{corollary}[theorem]{Corollary}

\theoremstyle{definition}

\numberwithin{equation}{section} 

\def\t{\tilde}

\begin{document}

\title[Mizuno-type result and Wallis' formula]{Mizuno-type result and Wallis' formula}

\author{Su Hu}
\address{Department of Mathematics, South China University of Technology, Guangzhou, Guangdong 510640, China}
\email{mahusu@scut.edu.cn}

\author{Min-Soo Kim}
\address{Department of Mathematics Education, Kyungnam University, Changwon, Gyeongnam 51767, Republic of Korea}
\email{mskim@kyungnam.ac.kr}


\subjclass[2010]{11M35, 33B15}
\keywords{Gamma function, Hurwitz zeta function, Wallis formula.}

\begin{abstract}Let $\t\Gamma(z)$ be the modified gamma function introduced  by the authors in a recent  preprint  ``arXiv2106.14674". 
In this note,
we obtain the following Mizuno-type result: 
\begin{equation*}
\prod_{m=0}^{\infty}\left\{\prod_{j=1}^{n}(m+z_{j})\right\}^{(-1)^{m}}=\frac{\left(\sqrt{\frac{\pi}{2}}\right)^n}{\prod_{j=1}^{n}\tilde\Gamma(z_{j})},
\end{equation*}
which imply a  Kurokawa--Wakayama type formula
\begin{equation*}
\prod_{m=0}^\infty\left((m+x)^{n}-y^n\right)^{(-1)^{m}}
 =\frac{\left(\sqrt{\frac{\pi}{2}}\right)^n}{\prod_{\zeta^{n}=1}\tilde\Gamma(x-\zeta y)}
\end{equation*}
and a Lerch-type formula
\begin{equation*}
\prod_{m=0}^\infty(m+x)^{(-1)^{m}}=\frac{\sqrt{\frac{\pi}{2}}}{\tilde\Gamma(x)}.
\end{equation*}
By setting $x=1$ in the above result, we recover Wallis' 1656 fomula
\begin{equation*}\frac{2\cdot2}{1\cdot 3}\frac{4\cdot4}{3\cdot 5}\frac{6\cdot6}{5\cdot 7}\cdots=\frac{\pi}{2}.
\end{equation*}
\end{abstract}

\maketitle

\def\ord{\text{ord}_p}
\def\ordt{\text{ord}_2}
\def\o{\omega}
\def\la{\langle}
\def\ra{\rangle}
\def\Log{{\rm Log}\, \Gamma_{p,N}^*}
\def\ov{\bar}

\section{Introduction}
\subsection{Lerch's formula} 
Lerch's 1894 formula asserts that 
\begin{equation}\label{Lerch}
\prod_{m=0}^\infty(m+x)=\frac{\sqrt{2\pi}}{\Gamma(x)}
\end{equation}
in the sense of zeta regularization, where $\Gamma(x)$ denotes the Euler gamma function.

Letting $x=1$ in (\ref{Lerch}) we get the following interesting result
\begin{equation}\label{Riemann}
\prod_{n=1}^\infty n=\sqrt{2\pi}.
\end{equation}
In \cite[Corollary 9.13]{KKS} the authors named its Riemann's formula, because it comes from Riemann's result in 1859:
\begin{equation}\label{Riemann2}
\zeta^{\prime}(0)=-\frac{1}{2}\log(2\pi).
\end{equation}

 Lerch himself extended (\ref{Lerch})  to the Gaussian quadratic field $\mathbb{Q}(i)$ as follows:
\begin{equation}\label{Lerch2}
\prod_{m=0}^\infty \left((m+x)^2+y^2\right)=\frac{2\pi}{\Gamma(x+iy)\Gamma(x-iy)}.
\end{equation}
Then in 2004 Kurokawa and Wakayama \cite{KW} proved the following generalization of (\ref{Lerch2}) to any cyclotomic fields $\mathbb{Q}(\zeta)$,
where $\zeta$ denotes the $n$th roots of unity:
\begin{equation}\label{KW}
\prod_{m=0}^{\infty}\left((m+x)^{n}-y^n\right)=\frac{(\sqrt{2\pi})^{n}}{\prod_{\zeta^{n}=1}\Gamma(x-\zeta y)}
\end{equation}
and in 2006, by applying Stark's summation formula \cite{Stark}, Mizuno \cite{Mizuno} got  a general form:
\begin{equation}\label{Mizuno}
\prod_{m=0}^{\infty}\left(\prod_{j=1}^{n}(m+z_{j})\right)=\frac{(\sqrt{2\pi})^{n}}{\prod_{j=1}^{n}\Gamma(z_{j})}
\end{equation}
for $z_{j}\in\mathbb{C}\setminus\{0,-1,-2, \ldots\}.$

\subsection{Our results}
In a recent preprint \cite{arXiv2}, we showed a close connection between (\ref{Riemann2}) and the following formula found by John Wallis in 1656:
\begin{equation}\label{Wallis} 
\frac{2\cdot2}{1\cdot 3}\frac{4\cdot4}{3\cdot 5}\frac{6\cdot6}{5\cdot 7}\cdots=\frac{\pi}{2}.
\end{equation}
 That  is, letting  $z\in\mathbb{C}\setminus\{0,-1,-2, \ldots\},$
 we showed that  
(\ref{Riemann2}) can be proved from the analytic properties for the Barnes' multiple zeta function:
\begin{equation}\label{M-Z}
\zeta_{N}(s,z)=\sum_{m_1,\ldots,m_N=0}^\infty
\frac{1}{(z+m_1+\cdots+m_N)^s},\quad\text{Re}(s)>N, 
\end{equation} 
while (\ref{Wallis}) is implied by its alternating form
\begin{equation}\label{M-Z-2}
\zeta_{E, N}(s,z)=\sum_{m_1,\ldots,m_N=0}^\infty
\frac{(-1)^{m_1+\cdots+m_N}}{(z+m_1+\cdots+m_N)^s},\quad\text{Re}(s)>0.
\end{equation} 

Let $\tilde\Gamma(z)$ denotes the modified gamma function, which is defined by the authors in \cite{arXiv3} from the  alternating Hurwitz zeta functions. In  subsection \ref{Sgamma} below,
we will give a brief review for its definition and properties by comparing with the ordinary gamma function $\Gamma(z)$.

In this note, inspiring by the above considerations, we extend Wallis' formula (\ref{Wallis}) as the following general form, which is an analogue of Mizuno's formula (\ref{Mizuno}) above.
\begin{theorem}\label{main theorem}
For $z_{j}\in\mathbb{C}\setminus\{0,-1,-2, \ldots\},$ we have
\begin{equation}\label{main}
\prod_{m=0}^{\infty}\left\{\prod_{j=1}^{n}(m+z_{j})\right\}^{(-1)^{m}}=\frac{\left(\sqrt{\frac{\pi}{2}}\right)^n}{\prod_{j=1}^{n}\tilde\Gamma(z_{j})}.
\end{equation}
 \end{theorem} 
 For $n=1$, letting $z_{1}=x$ in (\ref{main}) we get the following analogue of Lerch's formula (\ref{Lerch}). 
 \begin{corollary}[Lerch type formula] \label{Lerch type}
 \begin{equation}\label{Lerch type}
\prod_{m=0}^\infty(m+x)^{(-1)^{m}}=\frac{\sqrt{\frac{\pi}{2}}}{\tilde\Gamma(x)}.
\end{equation}
\end{corollary}  
So setting $x=1$ in (\ref{Lerch type}), by Lemma \ref{lemma2} below and taking squares on the both sides, we recover Wallis' formula (\ref{Wallis}).

For $n=2,$ letting $z_{1}=x+iy$ and $z_{2}=x-iy$ in (\ref{main}), we obtain an analogue of Lerch's formula  (\ref{Lerch2}).
\begin{corollary}[Lerch type formula in $\mathbb{Q}(i)$]
\begin{equation}\label{Lerch type 2}
\begin{aligned}
\prod_{m=0}^\infty \left((m+x)^2+y^2\right)^{(-1)^{m}}&=\prod_{m=0}^\infty\left((m+x+iy)(m+x-iy)\right)^{(-1)^{m}}\\
&=\frac{\frac{\pi}{2}}{\tilde\Gamma(x+iy)\tilde\Gamma(x-iy)}.
\end{aligned}
\end{equation}
\end{corollary}
For $n$th roots of unity $\zeta$, letting $z_{j}=x-\zeta^{j}y, (j=0,1,\ldots, n-1)$ in (\ref{main}) we get the following analogue of 
Kurokawa and Wakayama's formula (\ref{KW}).
\begin{corollary}[Kurokawa--Wakayama type formula]
\begin{equation}\label{Lerch type 3}
\begin{aligned}
\prod_{m=0}^\infty\left((m+x)^{n}-y^n\right)^{(-1)^{m}}&=\prod_{m=0}^{\infty}\left\{\prod_{j=0}^{n-1}(m+x-\zeta^{j}y)\right\}^{(-1)^{m}}\\
&=\frac{\left(\sqrt{\frac{\pi}{2}}\right)^n}{\prod_{\zeta^{n}=1}\tilde\Gamma(x-\zeta y)}.
\end{aligned}
\end{equation}
\end{corollary}

\subsection{Gamma function and Euler's constant}\label{Sgamma}
The Hurwitz zeta function 
\begin{equation}\label{Hurwitz}
\zeta(s,z)=\sum_{m=0}^\infty\frac{1}{(m+z)^{s}}\end{equation}
can be viewed as a source for several  special functions and mathematical constants.
Setting $z=1$ in (\ref{Hurwitz}), it reduces to the Riemann zeta function
\begin{equation}~\label{Ri-zeta}
\zeta(s)=\sum_{m=1}^{\infty}\frac{1}{m^{s}}.
\end{equation}

The generalized Stieltjes constant $\gamma_{k}(z)$ comes from the following  Laurent series expansion of $\zeta(s,z)$ around $s=1$
\begin{equation}\label{Stieltjes constant}
\zeta(s,z)=\frac{1}{s-1}+\sum_{k=0}^{\infty}\frac{(-1)^{k}\gamma_{k}(z)}{k!}(s-1)^{k}
\end{equation}
and $\gamma_{k}=\gamma_{k}(1)$ is the original Stieltjes constant in 1885 (see Stieltjes' original article \cite{St} and  Ferguson \cite{Fe}).
Letting $k=0$, we get Euler's constant
\begin{equation}
\begin{aligned}\gamma&:=\gamma_{0}(1)=\lim_{s\to1}\left(\zeta(s)-\frac{1}{s-1}\right)\\
&=\lim_{\alpha\to\infty}\left(\sum_{n=1}^{\alpha}\frac{1}{n}-\log \alpha\right)=0.5772156649\cdots.
\end{aligned}
\end{equation}

The gamma function $\Gamma(z)$ is defined by Euler from its integral representation 
\begin{equation}\label{Gamma} \Gamma(z)=\int_{0}^{\infty}t^{z-1}e^{-t}dt, \quad\textrm{Re}(z)>0,\end{equation}
 but it can also be defined from the derivatives of $\zeta(s,z)$ as follows
(e.g., \cite[Definition 9.6.13(1)]{Cohen}),
\begin{equation}\Gamma(z)=\exp\left(\zeta'(0,z)-\zeta'(0,1)\right)=\exp\left(\zeta'(0,z)-\zeta'(0)\right).\end{equation}
The following Weierstrass--Hadamard product  representation of $\Gamma(z)$ is well-known:
\begin{equation}\label{Hadamard}
\Gamma(z)=\frac{1}{z}e^{-\gamma z} \prod_{m=1}^{\infty}\left(e^{\frac{z}{m}}\left(1+\frac{z}{m}\right)^{-1}\right),
\end{equation}
where $\gamma$ is  Euler's constant.

Then the digamma functions can be defined from the derivatives  of the log gamma functions $\log\Gamma(z),$
that is,  \begin{equation}\label{Classical2} \psi(z):=\frac{d}{dz}\log\Gamma(z)\end{equation}
and more generally
 \begin{equation}\psi^{(n)}(z):=\left(\frac{d}{dz}\right)^n\psi(z),\quad n=0,1,2,\ldots\end{equation}
 (see \cite[p. 33]{SC}), 
and \cite[p. 33, Eq. (53)]{SC} shows that
\begin{equation}\label{Classical} \psi^{(n)}(z)=(-1)^{n+1}n!\zeta(n+1,z), \quad n = 1,2,\ldots
\end{equation}
(also see \cite[Proposition 9.6.41]{Cohen}).

Now we go to the alternating case. For details, we refer to \cite{arXiv3}. 
Let 
\begin{equation}\label{AHurwitz}
\zeta_{E}(s,z)=\sum_{m=0}^\infty\frac{(-1)^{m}}{(m+z)^{s}}\end{equation}
be the alternating Hurwitz zeta function.
Setting $z=1$ in (\ref{AHurwitz}), it reduces to Dirichlet's eta function
\begin{equation}~\label{Ezeta}
\eta(s)=\sum_{m=1}^{\infty}\frac{(-1)^{m+1}}{m^{s}}.
\end{equation}
According to Weil's history~\cite[p.~273--276]{Weil} (also see a survey by Goss~\cite[Section 2]{Goss}),
Euler used (\ref{Ezeta}) to ``prove"
\begin{equation}~\label{fe}
\frac{\eta(1-s)}{\eta(s)}=-\frac{\Gamma(s)(2^{s}-1)\textrm{cos}(\pi s/2)}{(2^{s-1}-1)\pi^{s}},
\end{equation}
which leads to the functional equation of the Riemann zeta function $\zeta(s)$.

As a result of analytic continuation, we see that $\zeta_E(s,z)$ is non-singular at $s=1.$  Thus we can  designate  a modified Stieltjes constant $\tilde\gamma_k(z)$ from
the Taylor expansion of $\zeta_E(s,z)$ at $s=1$,
\begin{equation}\label{l-s-con}
\zeta_E(s,z)=\sum_{k=0}^\infty\frac{(-1)^k\tilde\gamma_k(z)}{k!}(s-1)^k.
\end{equation}
In analogy with the classical case (\ref{Stieltjes constant}),  $\tilde\gamma_{k}=\tilde\gamma_{k}(1)$ is named the modified Stieltjes constant.
Letting $k=0$, we get the modified Euler  constant
\begin{equation}\label{gamma0}
\tilde\gamma_{0}:=\tilde\gamma_{0}(1)=\frac12+\frac12\sum_{j=1}^\infty(-1)^{j+1}\frac{1}{j(j+1)}.
 \end{equation}
 (See \cite[p. 4]{arXiv3}).
 
Following \cite[Proposition 2]{WZ},  the modified digamma function $\tilde\psi(z)$ is defined to be
\begin{equation}\label{psi-def}
\tilde\psi(z):=-\tilde\gamma_0(z).
\end{equation}
or equivalently 
\begin{equation}\label{ps-ga}
\begin{aligned}
\tilde\psi(z)&=-\frac{\Gamma'(z)}{\Gamma(z)}+\frac{\Gamma'(z/2)}{\Gamma(z/2)}+\log2
\\&=-\psi(z)+\psi(z/2)+\log2,
\end{aligned}
\end{equation}
where $\Gamma$ is the gamma function and  $\psi$ is the digamma function.
Let  \begin{equation}\label{poly-ga-def}
\t\psi^{(n)}(z):=\left(\frac{d}{dz}\right)^n\t\psi(z),\quad n=0,1,2,\ldots.
\end{equation}
As in the classical situation (\ref{Classical}), we have the following representation
\begin{equation}\label{ga-poly}
\t\psi^{(n)}(z)=(-1)^{n+1}n!\zeta_E(n+1,z), \quad n = 0,1,2,\ldots
\end{equation}
(see \cite[p. 957, 8.374]{GR}).

Inspiring by the classical formula (\ref{Classical2}), we define the modified gamma function $\tilde\Gamma(z)$ from the differential equation 
\begin{equation}
\tilde\psi(z)=\frac{d}{dz}\log\tilde\Gamma(z),\quad\textrm{Re}(z)>0
\end{equation}
and the following   analogue of the Weierstrass--Hadamard product (\ref{Hadamard}) has been shown in \cite[Theorem 1.12]{arXiv3}:
\begin{equation}\label{WH} 
\tilde\Gamma(z)=\frac1z e^{\tilde\gamma_0 z}\prod_{m=1}^\infty\left(e^{-\frac zm}\left(1+\frac zm\right)\right)^{(-1)^{m+1}},\end{equation}
where $\tilde\gamma_0$ is the modified Euler constant (see (\ref{gamma0})). 

The following two lemmas on the properties of the Dirichlet's eta function $\eta(s)$ and the modified gamma function $\tilde\Gamma(z)$ shall be used
in the proof of the main result.
\begin{lemma}\label{lemma1}   
\begin{itemize}
 \item [(1)] 
 \begin{equation}
\eta(1)=\tilde\gamma_{0}.
\end{equation}
 \item [(2)] 
 \begin{equation}
\eta^{\prime}(0)=\log\sqrt{\frac{\pi}{2}}.
\end{equation}\end{itemize}
\end{lemma}
\begin{proof} (1) Setting $s=1$ and $z=1$ in (\ref{l-s-con}), by (\ref{Ezeta}) and (\ref{gamma0}) we \begin{equation*} 
\eta(1)=\zeta_E(1,1)=\tilde\gamma_{0}(1)=\tilde\gamma_{0}.
\end{equation*}

(2) Since  $$\eta(s)=\left(1-2^{1-s}\right)\zeta(s),$$
by taking the derivatives on the both sides of the above equality and noticing that 
$$\zeta(0)=-\frac{1}{2}$$ (\cite[Theorem 12.16]{Apostol}) and  $$\zeta^{\prime}(0)=-\frac{1}{2}\log(2\pi),$$
we get $$\eta^{\prime}(0)=\log\sqrt{\frac{\pi}{2}},$$
which is what we want.\end{proof}

\begin{lemma}\label{lemma2}
 $$\tilde\Gamma(1)=\frac{\pi}{2}.$$ 
 \end{lemma}
 \begin{proof}
By \cite[p. 19, Eq. (2.31)]{arXiv3}, we have
\begin{equation}\label{lgamma} 
\log\t\Gamma(z)=-\log z+\t\gamma_0 z+\sum_{k=1}^\infty(-1)^{k}\left(\frac zk-\log\left(1+\frac zk\right)\right).
\end{equation}
Letting $z=1$ in (\ref{lgamma}), we see
\begin{equation}\label{lgamma2} 
\log\t\Gamma(1)=\t\gamma_0+\sum_{k=1}^\infty(-1)^{k}\left(\frac 1k-\log\left(1+\frac 1k\right)\right).
\end{equation}
By Lemma \ref{lemma1}(1) we have  $$\sum_{k=1}^\infty\frac{(-1)^{k}}{k}=-\eta(1)=-\t\gamma_{0},$$
so  (\ref{lgamma2}) implies
\begin{equation}\label{lgamma3} 
\log\t\Gamma(1)=\log\prod_{k=1}^{\infty}\left(\frac{k}{k+1}\right)^{(-1)^{k}}.
\end{equation} 
From Wallis' formula (\ref{Wallis}) we have
\begin{equation}\label{Wallis2}
\prod_{k=1}^{\infty}\left(\frac{k}{k+1}\right)^{(-1)^{k}}=\frac{\pi}{2}.
\end{equation}
Here it may be  necessary  to mention that  \cite[Section 3]{arXiv2} points out that Wallis  formula can be derived from the alternating multiple Hurwitz zeta functions (\ref{M-Z-2}) directly.
Substituting (\ref{Wallis2}) into (\ref{lgamma3}), 
we get $$ \log\t\Gamma(1)=\log \frac{\pi}{2}$$
and  $$\t\Gamma(1)=\frac{\pi}{2}.$$
which is what we want.
\end{proof}

\section{Proof of the main result}
In this section, we prove Theorem \ref{main theorem} by modifying the method of Mizuno \cite{Mizuno}.
The Weiestrass-Hadamard product of the modified gamma function $\tilde\Gamma(z)$ (\ref{WH}) will play
a key role in our approach.

Let $c\in\mathbb{N}_{0}=\mathbb{N}\cup\{0\}. $ Define
\begin{equation}
\Lambda_{c}^{*}(s)=\sum_{m=c+1}^{\infty} (-1)^{m+1} \prod_{j=1}^{n}(m+z_{j})^{-s},\quad\textrm{Re}(s)>0.
\end{equation}
It is easy to see that
\begin{equation}\label{(5)}
\begin{aligned}
&\quad\Lambda_{c}^{*}(s)-\sum_{m=c+1}^{\infty} (-1)^{m+1}m^{-ns}+s\left(\sum_{j=1}^{n}z_{j}\right)\sum_{m=c+1}^{\infty}(-1)^{m+1}m^{-(ns+1)}\\
&=\sum_{m=c+1}^{\infty} (-1)^{m+1}m^{-ns}\left\{\prod_{j=1}^{n}\left(1+\frac{z_{j}}{m}\right)^{-s}-1+s\left(\sum_{j=1}^{n}z_{j}\right)\frac{1}{m}\right\}.
\end{aligned}
\end{equation}
For $c\in\mathbb{N}$ big enough, we have $|\frac{z_{j}}{m}|<1$ for any $m\geq c+1$ and $j=1,2,\ldots,n$.
From the binomial theorem, 
\begin{equation}
\left(1+\frac{z_{j}}{m}\right)^{-s}=1-s\frac{z_{j}}{m}+O\left(\frac{1}{m^{2}}\right),
\end{equation}
so the right hand side of (\ref{(5)}) converges absolutely and uniformly on compact subset of 
$\{s\in\mathbb{C}: \textrm{Re}(s)>-\frac{1}{n}\}.$

Denote by 
\begin{equation}
\Lambda^{*}(s)=\sum_{m=0}^{\infty} (-1)^{m+1} \prod_{j=1}^{n}(m+z_{j})^{-s},\quad\textrm{Re}(s)>0.
\end{equation}
We have 
\begin{equation}\label{(6)}
\begin{aligned}
\Lambda^{*}(s)&=\sum_{m=0}^{c}(-1)^{m+1} \prod_{j=1}^{n}(m+z_{j})^{-s}+\Lambda_{c}^{*}(s)\\
&=\sum_{m=0}^{c}(-1)^{m+1} \prod_{j=1}^{n}(m+z_{j})^{-s}\\
&\quad+\sum_{m=c+1}^{\infty} (-1)^{m+1}m^{-ns}-s\left(\sum_{j=1}^{n}z_{j}\right)\sum_{m=c+1}^{\infty}(-1)^{m+1}m^{-(ns+1)}\\
&\quad+\sum_{m=c+1}^{\infty} (-1)^{m+1}m^{-ns}\left\{\prod_{j=1}^{n}\left(1+\frac{z_{j}}{m}\right)^{-s}-1+s\left(\sum_{j=1}^{n}z_{j}\right)\frac{1}{m}\right\}\\
&=\sum_{m=0}^{c}(-1)^{m+1} \prod_{j=1}^{n}(m+z_{j})^{-s}
+\left(\eta(ns)-\sum_{m=1}^{c} (-1)^{m+1}m^{-ns}\right)\\
&\quad -s\left(\sum_{j=1}^{n}z_{j}\right)\left(\eta(ns+1)-\sum_{m=1}^{c}(-1)^{m+1}m^{-(ns+1)}\right)\\
&\quad+\sum_{m=c+1}^{\infty} (-1)^{m+1}\left\{\prod_{j=1}^{n}\left(m+z_{j}\right)^{-s}-m^{-ns}+s\left(\sum_{j=1}^{n}z_{j}\right)m^{-(ns+1)}\right\}.
\end{aligned}
\end{equation}
Thus taking the derivatives on the both sides of the above equality, we have
\begin{equation}\label{(7)}
\begin{aligned}
\frac{\partial}{\partial s} \Lambda^{*}(s)\bigg|_{s=0}
&=-\sum_{m=0}^{c}(-1)^{m+1} \sum_{j=1}^{n}\log(m+z_{j})
\\&\quad+n\eta^{\prime}(0)-\sum_{m=1}^{c} (-1)^{m+1}(-n)\log m\\
&\quad-\left(\sum_{j=1}^{n}z_{j}\right)\eta(1)+\left(\sum_{j=1}^{n}z_{j}\right)\sum_{m=1}^{c}(-1)^{m+1}\frac{1}{m}\\
&\quad+\sum_{m=c+1}^{\infty} (-1)^{m+1}\left\{-\sum_{j=1}^{n}\log(m+z_{j})+n\log m+\left(\sum_{j=1}^{n}z_{j}\right)\frac{1}{m}\right\}\\
&=-\sum_{m=0}^{c}(-1)^{m+1} \sum_{j=1}^{n}\log(m+z_{j})+\sum_{j=1}^{n}\sum_{m=1}^{c}(-1)^{m+1} \left(\log m+\frac{z_{j}}{m}\right)\\
&\quad-\sum_{j=1}^{n}\sum_{m=c+1}^{\infty}(-1)^{m+1}\left\{\log\left(1+\frac{z_{j}}{m}\right)-\frac{z_{j}}{m}\right\}\\
&\quad+n\eta^{\prime}(0)-\left(\sum_{j=1}^{n}z_{j}\right)\eta(1).
\end{aligned}
\end{equation}
By Lemma \ref{lemma1} $\eta^{\prime}(0)=\log\sqrt{\frac{\pi}{2}}$ and $\eta(1)=\t\gamma_{0},$ we have
\begin{equation}\label{(8)}
\begin{aligned}
\frac{\partial}{\partial s} \Lambda^{*}(s)\bigg|_{s=0}
&=-\sum_{m=0}^{c}(-1)^{m+1} \sum_{j=1}^{n}\log(m+z_{j})\\
&\quad+\sum_{j=1}^{n}\sum_{m=1}^{c}(-1)^{m+1} \left(\log m+\frac{z_{j}}{m}\right)\\
&\quad-\sum_{j=1}^{n}\sum_{m=c+1}^{\infty}(-1)^{m+1}\left\{\log\left(1+\frac{z_{j}}{m}\right)-\frac{z_{j}}{m}\right\}\\
&\quad+n\log\sqrt{\frac{\pi}{2}}-\left(\sum_{j=1}^{n}z_{j}\right)\t\gamma_{0}.
\end{aligned}
\end{equation}
On the other hand,
since 
\begin{equation}
\Lambda^{*}(s)=\sum_{m=0}^{\infty} (-1)^{m+1} \prod_{j=1}^{n}(m+z_{j})^{-s},~\textrm{Re}(s)>0,
\end{equation}
 by taking the derivatives on the both sides directly,
 we have 
\begin{equation}\label{(9)}
\begin{aligned}
\frac{\partial}{\partial s} \Lambda^{*}(s)\bigg|_{s=0}
&=-\sum_{m=0}^{\infty}(-1)^{m+1} \sum_{j=1}^{n}\log(m+z_{j})\\
&=\sum_{m=0}^{\infty}\log\prod_{j=1}^{n}(m+z_{j})^{(-1)^{m}}.
\end{aligned}
\end{equation} 
Then comparing (\ref{(8)}) and (\ref{(9)}) we have
 \begin{equation}\label{(10)}
\begin{aligned}
\prod_{m=0}^{\infty}\left\{\prod_{j=1}^{n}(m+z_{j})\right\}^{(-1)^{m}}
&= \exp\left(\frac{\partial}{\partial s}\Lambda^{*}(s)\bigg|_{s=0}\right)\\
&=\prod_{j=1}^{n}\left\{\prod_{m=0}^{c}(m+z_{j})\right\}^{(-1)^{m}}\prod_{j=1}^{n}\prod_{m=1}^{c}m^{(-1)^{m+1}}\prod_{j=1}^{n}\prod_{m=1}^{c}e^{(-1)^{m+1}\frac{z_{j}}{m}}\\
&\quad\times\prod_{j=1}^{n}\left\{\prod_{m=c+1}^{\infty}\left(1+\frac{z_{j}}{m}\right)\right\}^{(-1)^{m}}\prod_{j=1}^{n}\prod_{m=c+1}^{\infty}e^{(-1)^{m+1}\frac{z_{j}}{m}}\\
&\quad\times\left(\sqrt{\frac{\pi}{2}}\right)^n\times   e^{\left(\sum_{j=1}^{n}z_{j}\right)(-\t\gamma_{0})}\\
&=\prod_{j=1}^{n} z_{j} \times \prod_{j=1}^{n}\left\{\prod_{m=1}^{c}(m+z_{j})\right\}^{(-1)^{m}}\\
&\quad\times\prod_{j=1}^{n}\prod_{m=1}^{c}m^{(-1)^{m+1}}\prod_{j=1}^{n}\prod_{m=1}^{c}e^{(-1)^{m+1}\frac{z_{j}}{m}}\\
&\quad\times\prod_{j=1}^{n}\left\{\prod_{m=c+1}^{\infty}\left(1+\frac{z_{j}}{m}\right)\right\}^{(-1)^{m}}\prod_{j=1}^{n}\prod_{m=c+1}^{\infty}e^{(-1)^{m+1}\frac{z_{j}}{m}}\\
&\quad\times\left(\sqrt{\frac{\pi}{2}}\right)^n\times   e^{\left(\sum_{j=1}^{n}z_{j}\right)(-\t\gamma_{0})}\\
&=\left(\sqrt{\frac{\pi}{2}}\right)^n
\prod_{j=1}^{n}\left\{ z_{j}e^{-\t\gamma_{0}z_{j}}\prod_{m=1}^\infty\left(e^{-\frac{z_{j}}{m}}\left(1+\frac{z_{j}}{m}\right)\right)^{(-1)^{m}}\right\}.\end{aligned}
\end{equation}  
Finally from the Weierstrass--Hadamard product representation of $\t\Gamma(z)$ (\ref{WH}) we get
\begin{equation}\label{main-proof}
\prod_{m=0}^{\infty}\left\{\prod_{j=1}^{n}(m+z_{j})\right\}^{(-1)^{m}}=\frac{\left(\sqrt{\frac{\pi}{2}}\right)^n}{\prod_{j=1}^{n}\tilde\Gamma(z_{j})},
\end{equation}
which is what we want.

\end{document}